\documentclass[a4paper,10pt]{scrartcl}
\usepackage[utf8]{inputenc}
\usepackage{amssymb,amsmath,amsthm}
\usepackage{array, hhline}
\usepackage{tikz}
\usepackage{authblk}

\usepackage[]{times}
\usepackage{fullpage}
\date{}
\usepackage{graphicx, epsfig}
\usepackage{amsmath}

\usepackage{amssymb}

\usepackage{tabularx, booktabs}
\usepackage{threeparttable}

\newcommand{\ds}{\displaystyle}
\newcommand{\ra}{\rangle}
\newcommand{\la}{\langle}

\def\l2{L^2(\Omega)}

\def\ha1{H^1(\Omega)}
\def\h10{H^1_0(\Omega)}

\def\la{\langle}
\def\ra{\rangle}

\def\ez{{\bf e_z}}
\def\bx{{\bf x}}

\def\R{{\rm I\kern-.17em R}}
\def\N{{\rm I\kern-.17em N}}

\newtheorem{remark}{\bf Remark}[section]
\newtheorem{theorem}{\bf Theorem}[section]

\begin{document}
\title{A study on iterative methods for solving Richards' equation}

\author[F. List and F.A. Radu]
       {Florian List$^\dag$ and Florin A. Radu\thanks{florin.radu@uib.no (corresponding author),
       \\{\small
        $^\dag$ University of Stuttgart, Department of Mathematics, Stuttgart, Germany\\
        $^\ast $  University of Bergen, Department of Mathematics, Bergen, Norway}}}

\maketitle

\begin{abstract}
{\bfseries Abstract.} This work concerns linearization methods for efficiently solving the Richards' equation, a degenerate elliptic-parabolic equation which models flow in saturated/unsaturated porous media. The discretization of Richards' equation is based on backward Euler in time and Galerkin finite elements in space. The most valuable linearization schemes for Richards' equation, i.e. the Newton method, the Picard method, the Picard/Newton method and the $L-$scheme are presented and their performance is comparatively studied. The convergence, the computational time and the condition numbers for the underlying linear systems are recorded. The convergence of the $L-$scheme is theoretically proved and  the convergence of the other methods is discussed. A new scheme is proposed, the $L-$scheme/Newton method which is more robust and quadratically convergent. The linearization methods are tested on illustrative numerical examples.
\end{abstract}

\noindent {{\bf Keywords} Richards' equation, linearization schemes, Newton method, Picard method, convergence analysis, flow in porous media, Galerkin finite elements.

\section{Introduction}

There are plenty of societal relevant applications of multiphase flow in porous media, e.g. water and soil pollution, CO$_2$ storage, nuclear waste management, or enhanced oil recovery, to name a few. Mathematical modelling and numerical simulations are powerful, well-recognised tools for predicting flow in porous media and therefore for understanding and finally solving problems like the ones mentioned above. Nevertheless, mathematical models for multiphase flow in porous media involve coupled, nonlinear partial differential equations on huge, complex domains and with parameters which may vary on multiple order of magnitudes. Moreover, typical for the type of applications we mentioned are long term time evolutions, recommending the use of implicit schemes which allow large time steps. Due to these, the design and analysis of appropriate numerical schemes for multiphase flow in porous media is a very challenging task. Despite of intensive research in the last decades, there is still a strong need for robust numerical schemes for multiphase flow in porous media.

In this work we consider a particular case of two-phase flow: flow of water in soil, including the region near the surface where the pores are filled with water and air (unsaturated zone). By considering that the pressure of air remains constant, i.e. zero, water flow through saturated/unsaturated porous media is mathematically described  by Richards' equation

\begin{equation}\label{EQ_BASIC_RICHARDS}
   \partial_t \theta (\Psi) - \nabla \cdot (K(\theta(\Psi)) \nabla (\Psi + z)) = 0,
\end{equation}
which has been proposed by L.A. Richards in 1930 (see e.g. \cite{Bear}). In (\ref{EQ_BASIC_RICHARDS}), $\Psi$ denotes the pressure head, $\theta$ the water content, $K$ stands for the hydraulic conductivity of the medium and $z$ for the height against the gravitational direction. Based on experimental results, different curves have been proposed for describing the dependency between $K$, $\theta$ and $\Psi$ (see e.~g. \cite{Bear}), yielding the nonlinear model (\ref{EQ_BASIC_RICHARDS}). In the saturated zone, i.e. where the pores are filled only with water, we have  $\theta$ and $K$ constants and $\Psi \ge 0$. Whenever the flow is unsaturated,  $\theta$ and $K$ are nonlinear, monotone functions and $\Psi < 0$. We point out that Richards' equation degenerates when $K(\theta(\Psi)) \rightarrow 0$ (slow diffusion case) or when  $\theta^\prime = 0$ (fast diffusion case). The regions of degeneracy depend on the saturation of the medium; therefore these regions are not known a-priori and may vary in time and space. In this paper we concentrate on the fast diffusion case, therefore Richards'  equation will be a nonlinear, degenerate elliptic-parabolic partial differential equation. Typically for this case is also the low regularity of the solution \cite{AltLuckhaus}. The nonlinearities and the degeneracy make the design and analysis of numerical schemes for the Richards' equation very difficult.

The first choice for the temporal discretization is the backward Euler method. There are two reasons for this: the need of a stable discretization allowing large time steps and the low regularity of the solution which does not support any higher order scheme. As regards the spatial discretization there are much more options possible. Galerkin finite elements were used in \cite{Arbogast1990,Arbogast1993,Ebmeyer,NochettoVerdi,Slodicka,Pop}, often together with mass lumping to ensure a maximum principle \cite{Celia}. Locally mass conservative schemes for Richards' equation were proposed and analysed in \cite{Eymard1999,Eymard2006} (finite volumes), in \cite{Klausen} (multipoint flux approximation) or \cite{ArbogastWheelerZhang,Woodward,Yotov,Bause,Radu2004,Radu2014} (mixed finite element method). The analysis is performed mostly by using the Kirchhoff transformation (which combines the two main nonlinearities in one) \cite{AltLuckhaus,ArbogastWheelerZhang,Woodward,Radu2004,Radu2008} or, alternatively, by restricting the applicability, e.g. to strictly unsaturated case \cite{Arbogast1993,Radu2014}. To deal with the low regularity of the solution, a time integration together with the use of Green's operator is usually necessary \cite{NochettoVerdi,ArbogastWheelerZhang}.

The systems to be solved in each time step after temporal and spatial discretization are nonlinear and one needs an efficient and robust algorithm to solve them. The main linearization methods which are used for the Richards equation are: Newton (also called Newton--Raphson in the literature) method, Picard method, modified Picard method, the $L-$scheme, and combination of them. The Newton method, which is quadratically convergent was very successfully applied to Richards' equation in e.g.  \cite{Putti,Radu2006,Park,Celia,LehmannAckerer}. The drawback of Newton's method is that it is only locally convergent and involves the computation of derivatives. Although the use of the solution of the last time step to start the Newton iterations improves considerably the robustness of Newton' method, in the degenerate case (saturated/unsaturated flow) the convergence is ensured only when a regularization step is applied and under additional constraints on the discretization parameters, see \cite{Radu2006} for details. The Picard method is, although widely used, not a good choice when applied to Richards' equation as clearly shown in \cite{Celia,LehmannAckerer}. In \cite{Celia} is proposed an improvement of the Picard method, resulting in a new method called modified Picard. This method coincides with the Newton method in the case of a constant conductivity $K$ or when applied to Richards' equation together with the Kirchhoff transformation. The modified Picard method is only linearly convergent, but more robust than Newton's method. An efficient combination of the modified Picard and the Newton method, the Picard/Newton method is proposed in \cite{LehmannAckerer}. The $L-$method is the only method which uses the monotonicity of $\theta(\cdot)$. It was proposed for Richards' equation in \cite{Yong,Slodicka,Pop2004}. The method is robust and linearly convergent, and does not involve the computation of any derivatives. Moreover, the convergence rate does not depend on the mesh size. In the case of a constant $K$ or when applied to Richards' equation together with the Kirchhoff transformation, one can improve the convergence of the $L-$method by adaptively computing $L$, this being the idea of the J\"ager-Ka\v{c}ur method \cite{Kacur}. The choice of the Jacobian matrix for $L$ would lead to Newton's method, therefore in this case all the three methods (Newton, modified Picard and $L-$scheme) will coincide. It is worth to mention that both the $L-$method and the modified Picard method can be seen as quasi-Newton (or Newton-like) methods. We refer to \cite{KnabnerBook} for a comprehensive presentation of Newton's method and its variants. For the sake of completeness we mention also the semi-smooth Newton method \cite{Serge} and $L-$method \cite{Radu2015} as valuable linearization methods for two-phase flow in porous media.


In this paper we concentrate on linearization methods for Richards' equation. The new contributions of this paper are:
\begin{itemize}
\item{A comprehensive study on the most valuable linearization methods for Richards' equation: the Newton method, the modified Picard, the Picard/Newton method and the $L-$scheme. The study includes numerical convergence, CPU time and condition number of the resulting linear systems. }
\item{The design of a new scheme based on the $L-$scheme and Newton's method, the  $L-$scheme/Newton method which is robust and quadratically convergent.}
\item{Provide the theoretical proof for the convergence of the $L-$scheme for Richards' equation (without using the Kirchhoff transformation) and discuss the convergence of modified Picard and Newton methods. The analysis furnishes new insights and helps towards a deeper understanding of the linearization schemes.}
\end{itemize}

The present paper can be seen as a continuation of the works \cite{Celia}  and \cite{LehmannAckerer}, and it is written in a similar spirit. We added in the study the $L-$schemes (including a new scheme combining Newton's method with the $L-$scheme) and we focus now on 2D numerical results (the two mentioned papers based their conclusions on 1D simulations). We present illustrative examples, with realistic parameters so that the computations are relevant for practical applications.


The paper is structured as follows. In the next section the variational formulations (continuous and fully discrete) of Richards' equation are presented together with the considered linearization schemes. In Section \ref{sec:convergence} we discuss the theoretical convergence of the methods. The next section concerns the numerical results. The paper is ending by a concluding section.

\section{Linearization methods for Richards' equation}\label{sec:linearization}

Throughout this paper we use common notations in the functional analysis. Let $\Omega$ be a bounded domain in $\R^d, d = 1, 2\,  \rm{ or } \, 3$, having a Lipschitz continuous boundary $\partial \Omega$. We denote by $L^2(\Omega)$ the space of real valued square integrable functions defined on $\Omega$, and by $H^1(\Omega)$ its subspace containing functions having also the first order derivatives in $L^2(\Omega)$. Let $H_0^1(\Omega)$ be the space of functions in  $H^1(\Omega)$ which vanish on $\partial \Omega$, and $H^{-1}(\Omega)$ its dual. Further, we denote by  $\la\cdot,\cdot \ra$ the inner product on $L^2(\Omega)$ or the duality product $H_0^1(\Omega), H^{-1}(\Omega)$, and by $\ds \| \cdot \|$ the norm of $L^2(\Omega)$. $L_f$ stays for the Lipschitz constant of a Lipschitz continuous function $f(\cdot)$.


We consider to solve the Richards' equation \eqref{EQ_BASIC_RICHARDS} on  $(0, T] \times \Omega$, with $T$ denoting the final time and with homogeneous Dirichlet boundary conditions and an initial condition given by $\Psi(0, \bx) = \Psi^0 (\bx)$ for all $\bx \in \Omega$. We will use linear Galerkin finite elements for this study, but the linearization methods considered can be applied to any discretization method. We restrict the formulations and analysis to homogeneous Dirichlet boundary conditions just for the sake of simplicity, the extension to more general boundary conditions being straightforward (the numerical examples in Section \ref{sec:numerics} involve general boundary conditions). The continuous Galerkin formulation of  \eqref{EQ_BASIC_RICHARDS} reads as:

Find $\Psi \in H_0^1(\Omega)$ such that there holds
\begin{equation}
\la \partial_t \theta(\Psi), \phi \ra + \la K(\theta(\Psi)) \left(\nabla \Psi + \ez \right), \nabla \phi \ra = \la f, \phi \ra,\label{eq:richards_weak}
\end{equation}
for all $\phi \in H_0^1(\Omega)$, with $\ez := \nabla z$. Results concerning the existence and uniqueness of solutions of \eqref{eq:richards_weak} can be found in several papers, e.g. \cite{AltLuckhaus}.

For the discretization in time we let $N \in \mathbb N$ be strictly positive, and define the time step $\tau = {T/N}$, as well as $t_n = n\tau$ $(n \in \{1, 2, \dots, N\})$. Furthermore, ${\mathcal{T}_h}$ is a regular decomposition of $\Omega \subset {\R}^d$ into closed $d$-simplices; $h$ stands for the mesh diameter. Here we assume $\overline\Omega = \cup_{T \in \mathcal T_h} T$, hence $\Omega$ is assumed polygonal. The Galerkin finite element space is given by
\begin{equation}
V_h := \{ v_h \in H_0^1(\Omega) | \, {v_h}_{|_{T}} \in \mathcal{P}_1(T), T \in \mathcal{T}_h\},
\label{eq:def_discrete_subspace}
\end{equation}
where $\mathcal{P}_1(T)$ denotes the space of linear polynomials on any simplex $T$. For details about this finite element space and the implementation we refer to standard books, like e.g. \cite{KnabnerBook}.

By using the backward Euler method in time and the linear Galerkin finite elements defined above in space, the fully discrete variational formulation of \eqref{eq:richards_weak} at time $t_n$ reads as:

Let $n \in \{1, \ldots, N\}$ and $\Psi_h^{n-1} \in V_h$ be given. Find $\Psi_h^n \in V_h$ such that there holds
\begin{equation}\label{eq:discrete_formulation}
\la \theta(\Psi_h^{n})-\theta(\Psi_h^{n-1}), v_h \ra + \tau \la K(\Theta(\Psi_h^{n})) \left(\nabla \Psi_h^{n} + \vec{e}_z\right), \nabla v_h \ra = \tau \la f^{n}, v_h \ra,
\end{equation}
for all $v_h \in V_h$. At the first time step we take $\Psi_h^{0} = P_h \Psi^0 \in V_h$, with $P_h:  H_0^1(\Omega) \rightarrow V_h$ being the standard projection. We assume in the next that the fully discrete schemes above have a unique solution and we refer to \cite{Arbogast1990,Arbogast1993,Ebmeyer,Pop} for a proof. 

At this point, dealing with the doubly nonlinear character of Richards' equation due to the relations $K(\theta)$ and $\theta(\Psi)$ is essential. We will briefly present in the following the main linearization methods used to solve the nonlinear problem \eqref{eq:discrete_formulation}: the Newton method, the modified Picard method (called simply Picard's method below, when does not exist a possibility of confusion) and the $L-$schemes.

We denote the discrete solution at time level $n$ (which is now fixed) and iteration $j \in \mathbb{N}$ by $\Psi_h^{n,j}$ henceforth. The iterations are always starting with the solution at the last time step, i.e. $\Psi_h^{n,0} = \Psi_h^{n-1}$. The Newton method to solve \eqref{eq:discrete_formulation} reads as:

Let $\Psi_h^{n-1}, \Psi_h^{n,j-1} \in V_h$ be given. Find $\Psi_h^{n,j} \in V_h$, so that 
\begin{equation}\label{eq:Newton}
\begin{array}{c}
\ds \la \theta(\Psi_h^{n,j-1}), v_h \ra + \la \theta'(\Psi_h^{n,j-1})(\Psi_h^{n,j}-\Psi_h^{n,j-1}), v_h \ra + \tau \la K(\Psi_h^{n,j-1}) (\nabla \Psi_h^{n,j} + \ez), \nabla v_h \ra \\[2ex]
 \ds + \tau \la K'(\Psi_h^{n,j-1}) (\nabla \Psi_h^{n,j-1} + \ez)(\Psi_h^{n,j}-\Psi_h^{n,j-1}), \nabla v_h \ra = \tau \la f^{n}, v_h \ra  + \la \theta(\Psi_h^{n-1}), v_h \ra,
\end{array}
\end{equation}
holds for all $v_h \in V_h$. Newton's method is quadratically, but only locally convergent. As mentioned above, although $\Psi_h^{n,0}:=\Psi_h^{n-1}$ might be an appropriate choice, failure of Newton's method can occur (see \cite{Radu2006} and the numerical examples given below). 

The modified Picard method was proposed by \cite{Celia} and reads as:

Let $\Psi_h^{n-1}, \Psi_h^{n,j-1} \in V_h$ be given. Find $\Psi_h^{n,j} \in V_h$, so that 
\begin{equation}\label{eq:Picard}
\begin{array}{l}
\ds \la \theta(\Psi_h^{n,j-1}), v_h\ra + \la \theta'(\Psi_h^{n,j-1})(\Psi_h^{n,j}-\Psi_h^{n,j-1}), v_h \ra \\[2ex]
\ds \quad + \tau \la K(\Psi_h^{n,j-1}) (\nabla \Psi_h^{n,j}+ \ez), \nabla v_h \ra = \tau \la f^{n}, v_h \ra + \la \theta(\Psi_h^{n-1}), v_h\ra,
\end{array}
\end{equation}
holds for all $v_h \in V_h$. The modified Picard method was shown to perform much better than the classical Picard method \cite{Celia,LehmannAckerer}. The idea is to discretize the time nonlinearity quadratically, whereas the nonlinearity in $K$ is linearly approximated. The method is therefore linearly convergent. The method still involves the computation of derivatives and in the degenerate case might also fail to converge (see the numerical examples in Section \ref{sec:numerics}).

The $L-$method was proposed for Richards' equation by \cite{Yong,Slodicka,Pop2004} and it is the only method which exploits the monotonicity of $\theta$. The $L-$scheme to solve the nonlinear problem  \eqref{eq:discrete_formulation} reads as:

Let $\Psi_h^{n-1}, \Psi_h^{n,j-1} \in V_h$ and $L > 0$ be given. Find $\Psi_h^{n,j} \in V_h$, so that
\begin{equation}\label{eq:L-scheme}
\begin{array}{l}
\la \theta(\Psi_h^{n,j-1}), v_h \ra + L \la \Psi_h^{n,j}-\Psi_h^{n,j-1}, v_h \ra \\[2ex]
\quad + \tau \la K(\Psi_h^{n,j-1}) (\nabla \Psi_h^{n,j} + \ez), \nabla v_h \ra = \tau \la f^{n}, v_h \ra + \la \theta(\Psi_h^{n-1}), v_h\ra,
\end{array}
\end{equation}
holds for all $v_h \in V_h$. To ensure the convergence of the scheme, the constant  $L$ should  satisfy $L \geq L_\theta (=: \sup_{\Psi} |\theta'(\Psi)|$) (see Section \ref{sec:convergence} for details). The $L-$scheme is robust and linearly convergent. Furthermore, the scheme does not involve the computation of any derivatives. The key element in the new scheme is the addition of a stabilisation term  $L \la \Psi_h^{n,j}-\Psi_h^{n,j-1}, v_h \ra $, which together with the monotonicity of $\theta$ will ensure the convergence of the scheme. 

\begin{remark} It is to be seen that in the case of a constant K, the methods \eqref{eq:Newton} and  \eqref{eq:Picard} coincide. Moreover, if $L$ is replaced by the Jacobian matrix in \eqref{eq:L-scheme} one obtains again the modified Picard scheme \eqref{eq:Picard}.
\end{remark}

Any of the linearization methods presented above leads to a system of linear equations for  $\Psi_h^{n,j}$ (more precisely, the unknown will be the vector with the components of $\Psi_h^{n,j}$ in a basis of $V_h$). The derivatives of the water content and the hydraulic conductivity in case of the modified Picard scheme and Newton's method can be computed analytically or by a perturbation approach as suggested by \cite{Forsyth} and occurring integrals are approximated by a quadrature formula. 

For stopping the iterations, we adopt a general criterion for convergence given by
\begin{equation}
\|\Psi_h^{n,j}-\Psi_h^{n,j-1}\| \leq \varepsilon_{a} + \varepsilon_{r} \|\Psi_h^{n,j}\|,
\label{eq:criterion}
\end{equation}
with the Euclidean norm $\|\cdot\|$ and some constants $\varepsilon_{a} > 0$ and $\varepsilon_{r} > 0$. The tolerances $\varepsilon_a$ and $\varepsilon_r$ in criterion \eqref{eq:criterion} are both taken as $10^{-5}$ in all numerical simulations in this paper as proposed by \cite{LehmannAckerer}. We refer to \cite{Huang} for possible improvements of the stopping criterion.

The Newton method is the only method out of the proposed three which is second order convergent. Nevertheless, it is not that robust as the other, linearly convergent, methods. In order to increase the robustness of Newton's method one can perform first a few (modified) Picard iterations, this being the combined Picard/Newton scheme proposed in \cite{LehmannAckerer}. The Picard/Newton method is shown to perform better than both the Newton and the modified Picard method \cite{LehmannAckerer}. We propose in this paper also a combination of the $L-$scheme with the Newton method, the $L-$scheme/Newton method. The mixed methods are based upon the idea to harness the robustness of the $L-$scheme or the modified Picard scheme initially and to switch to Newton's method e.g. if 
\begin{equation}
\|\Psi_h^{n,j}-\Psi_h^{n,j-1}\| \leq \delta_{a} + \delta_{r} \|\Psi_h^{n,j}\|,
\label{eq:criterion_switch}
\end{equation}
is satisfied for $\delta_{a}, \delta_{r} > 0$, similar to criterion \eqref{eq:criterion}. However, an appropriate choice of the parameters $\delta_{a}, \delta_{r}$ is intricate and heavily dependent on the problem, for which reason a switch of the method after a fixed number of iterations may be an alternative. As shown in Section \ref{sec:numerics} this new method incorporating the $L-$scheme seems to perform best with respect to computing time and robustness.


\section{Convergence results}\label{sec:convergence}
In this section we will rigorously analyse the convergence of the $L-$scheme and discuss the convergence of Newton's and modified Picard's method. We denote by
\begin{equation}
e^{n,j} = \Psi^{n,j}_h- \Psi^{n}_h,
\end{equation}
the error at iteration $j$. A scheme is convergent if $e^{n,j} \rightarrow 0$, when $j \rightarrow \infty$.

The following assumptions on the coefficient functions and the discrete solution are defining the framework in which we can prove the convergence of the $L-$scheme.
\begin{itemize}
\item[(A1)]{The function $\theta(\cdot)$ is monotonically increasing and Lipschitz continuous. }
\item[(A2)]{The function $K(\cdot)$ is Lipschitz continuous and there exist two constants $K_m$ and $K_M$ such that $ 0 < K_m \le K(\theta) \le K_M < \infty, \, \, \forall \theta \in \R.$}
\item[(A3)]{The solution of problem \eqref{eq:discrete_formulation}  satisfies $\| \nabla \Psi^n_h \|_\infty \le M < \infty $, with $\| \cdot \|_\infty$ denoting the $L^\infty(\Omega)$-norm.}
\end{itemize}

We can now state the central theoretical result of this paper.

\begin{theorem}\label{theorem:convergenceLscheme}
Let $n \in \{1, \ldots, N\}$ be given and assume (A1) - (A3) hold true. If the constant $L$ and the time step are chosen such that \eqref{condition_convergence_Lscheme} is satisfied, the $L-$scheme \eqref{eq:L-scheme} converges linearly, with a rate of convergence given by
\begin{equation}\label{rate}
\sqrt{ \dfrac{L}{ L + \dfrac{K_m \tau}{C_\Omega^2} } }.
\end{equation}
\end{theorem}

\begin{proof} By subtracting \eqref{eq:discrete_formulation}   from  \eqref{eq:L-scheme} we obtain for any $v_h \in V_h$ and any $j \ge 1$
\begin{displaymath}
\begin{array}{l}
\ds \la \theta(\Psi_h^{n,j-1}) -  \theta(\Psi_h^{n}), v_h \ra + L \la e^{n,j}-e^{n,j-1}, v_h \ra \\[2ex]
\ds \quad + \tau \la K(\Psi_h^{n,j-1}) \nabla \Psi_h^{n,j} - K(\Psi_h^{n}) \nabla \Psi_h^{n}, \nabla v_h \ra +\tau  \la (K(\Psi_h^{n,j-1})  - K(\Psi_h^{n}) )\ez, \nabla v_h \ra = 0.
\end{array}
\end{displaymath}
By testing the above with $v_h = e^{n,j}$ and doing some algebraic manipulations one gets
\begin{equation}\label{eq:proof1}
\begin{array}{l}
\ds \la \theta(\Psi_h^{n,j-1}) -  \theta(\Psi_h^{n}), e^{n,j-1} \ra + \la \theta(\Psi_h^{n,j-1}) -  \theta(\Psi_h^{n}), e^{n,j}-e^{n,j-1} \ra   \\[2ex]
\ds \quad + \dfrac{L}{2} \|e^{n,j}\|^2   +  \dfrac{L}{2} \|e^{n,j}-e^{n,j-1}\|^2  -  \dfrac{L}{2} \|e^{n,j-1}\|^2   \\[2ex]
\ds \quad + \tau \la K(\Psi_h^{n,j-1}) \nabla e^{n,j}, \nabla e^{n,j} \ra + \la (K(\Psi_h^{n,j-1})  - K(\Psi_h^{n}) )\nabla \Psi_h^{n}, \nabla e^{n,j} \ra   \\[2ex]
\ds \quad  +\tau  \la (K(\Psi_h^{n,j-1})  - K(\Psi_h^{n}) )\ez, \nabla e^{n,j} \ra = 0,
\end{array}
\end{equation}
or, equivalently
\begin{equation}\label{eq:proof2}
\begin{array}{l}
\ds \la \theta(\Psi_h^{n,j-1}) -  \theta(\Psi_h^{n}), e^{n,j-1} \ra  + \dfrac{L}{2} \|e^{n,j}\|^2   +  \dfrac{L}{2} \|e^{n,j}-e^{n,j-1}\|^2  \\[2ex]
\ds \quad + \tau \la K(\Psi_h^{n,j-1}) \nabla e^{n,j}, \nabla e^{n,j} \ra   = \dfrac{L}{2} \|e^{n,j-1}\|^2  - \la \theta(\Psi_h^{n,j-1}) -  \theta(\Psi_h^{n}), e^{n,j}-e^{n,j-1} \ra \\[2ex]
\ds   \quad   - \la (K(\Psi_h^{n,j-1})  - K(\Psi_h^{n}) )\nabla \Psi_h^{n}, \nabla e^{n,j} \ra - \tau  \la (K(\Psi_h^{n,j-1})  - K(\Psi_h^{n}) )\ez, \nabla e^{n,j} \ra.
\end{array}
\end{equation}
By using now the monotonicity of $\theta(\cdot)$, its Lipschitz continuity (A1), the boundedness (from below) and Lipschitz continuity of $K(\cdot)$, i.e. (A2), the boundedness of  $\nabla \Psi_h^{n}$, and the Young and Cauchy-Schwarz inequalities one obtains from \eqref{eq:proof2}
\begin{equation}\label{eq:proof3}
\begin{array}{l}
\ds \dfrac{1}{L_\theta} \| \theta(\Psi_h^{n,j-1}) -  \theta(\Psi_h^{n})\|^2 + \dfrac{L}{2} \|e^{n,j}\|^2   +  \dfrac{L}{2} \|e^{n,j}-e^{n,j-1}\|^2  + \tau K_m \| \nabla e^{n,j} \|^2 \\[2ex]
\ds \quad  \le  \dfrac{L}{2} \|e^{n,j-1}\|^2  +  \dfrac{1}{2 L} \| \theta(\Psi_h^{n,j-1}) -  \theta(\Psi_h^{n}) \|^2 +   \dfrac{L}{2}  \|e^{n,j}-e^{n,j-1} \|^2 \\[2ex]
\ds   \quad \quad  + \dfrac{\tau (M+1)^2 L_K^2} {2 K_m}\| (\theta(\Psi_h^{n,j-1}) - \theta(\Psi_h^{n})) \|^2  + \dfrac{\tau K_m}{2} \| \nabla e^{n,j} \|^2 .
\end{array}
\end{equation}
After some obvious simplifications, the inequality \eqref{eq:proof3} becomes
\begin{equation}\label{eq:proof4}
\begin{array}{l}
\ds L\|e^{n,j}\|^2   +  \tau K_m\| \nabla e^{n,j} \|^2 \\[2ex]
\ds \quad  +  (\dfrac{2}{L_\Theta} - \dfrac{1}{L} - \dfrac{\tau (M+1)^2 L_K^2} {K_m}) \| \theta(\Psi_h^{n,j-1}) -  \theta(\Psi_h^{n})\|^2  \le  L \|e^{n,j-1}\|^2.
\end{array}
\end{equation}
Finally, by choosing $L >0 $ and the time step $\tau$ such  that 
\begin{equation}\label{condition_convergence_Lscheme}
\dfrac{2}{L_\theta} - \dfrac{1}{L} - \dfrac{\tau (M+1)^2 L_K^2} {K_m} \ge 0
\end{equation}
and by using the Poincare inequality (recall that $e^{n,j} \in H_0^1(\Omega) $)
\begin{equation}\label{poincare}
\| e^{n,j} \| \le C_\Omega \| \nabla e^{n,j} \|,
\end{equation}
from \eqref{eq:proof4} follows the convergence of the scheme \eqref{eq:L-scheme}
\begin{equation}\label{eq:proof5}
\begin{array}{l}
\ds \|e^{n,j}\|^2  \le \dfrac{L}{ L + \dfrac{K_m \tau}{C_\Omega^2} } \|e^{n,j-1}\|^2 .
\end{array}
\end{equation}

\end{proof}

We continue with some important remarks concerning the result above and the implications to the convergence of the Newton and modified Picard methods.

\begin{remark}\label{remark:noK} In the case of a constant hydraulic conductivity $K$ (or if we refer to Richards' equation after Kirchhoff's transformation and without gravity, see e.g. \cite{Pop2004}), the condition for convergence of the $L-$scheme \eqref{eq:L-scheme} simply becomes 
\begin{equation}\label{condition_constantK}
L \ge \dfrac{L_\theta}{2}
\end{equation}
and there is {\bf no restriction} on the time step size. Furthermore, the assumptions (A2) and (A3) are not necessary in this case.
\end{remark}

\begin{remark} \label{remark:rate} The  rate of convergence \eqref{rate} depends on $K_m$, $\tau$ and  $ L$, but it is independent of the mesh size. Smaller $L$ or larger time steps are resulting in a faster convergence. We also emphasise that larger hydraulic conductivities will imply a faster convergence as well.
\end{remark}

\begin{remark}\label{remark:restriction_tau} In the general case, the optimal choice is $L = L_\theta$ and $\tau = \dfrac{K_m}{L_\theta (M+1)^2 L_K^2}$. The restriction on the time step size (after choosing $L = L_\theta$) is $\tau \le \dfrac{K_m}{L_\theta (M+1)^2 L_K^2} $, which is relatively mild because it does not involve the mesh size or any regularization parameter.
\end{remark}

\begin{remark} \label{remark:global_convergence}The convergence of the $L-$scheme is global, i.e. independent of the initial choice. Nevertheless, it is obviously beneficial if one starts the iterations with the solution of the last time step.
\end{remark}

\begin{remark} \label{remark:convergenceNewton_constantK}The convergence of the modified Picard method and of the Newton method is studied in \cite{Radu2006} for the case of constant $K$ or for Richards' equation after Kirchhoff's transformation and without gravity. A regularization step is in this case necessary to ensure the convergence. The corresponding convergence condition to \eqref{condition_convergence_Lscheme} will look like
\begin{equation}\label{condition_Newton_constantK}
\tau \le C \epsilon^3 h^d,
\end{equation}
with $\epsilon$ denoting the regularization parameter and $h$ the mesh size, $d$ the spatial dimension and $C$ a constant not depending on the discretization parameters. A similar condition is derived also for the J\"ager-Ka\v{c}ur scheme, see \cite{Radu2006}. We remark that the condition \eqref{condition_Newton_constantK} is much more restrictive than the condition  \eqref{condition_convergence_Lscheme}. The proofs in \cite{Radu2006} are done for mixed finite element based discretizations, but the proof for Galerkin finite elements is similar. The condition \eqref{condition_Newton_constantK} is derived by using some inverse estimates and it is in practice quite pessimistic. Nevertheless, we emphasise the fact that the convergence is ensured only when doing a regularization step (reflected by the $\epsilon$ in  \eqref{condition_Newton_constantK}) and this is what one sees in practice as well (see Section \ref{sec:numerics}).
\end{remark}

\begin{remark} \label{remark:convergenceNewton}One can extend the convergence proof in \cite{Radu2006} for Newton's and modified Picard's methods to the general case of a nonlinear $K$ and saturated/unsaturated flow. Under a similar assumption (A2) for the modified Picard and an assumption involving also the Lipschitz continuity of the derivative of $K$ for Newton's method one can show the convergence of the methods. The modified Picard will be linearly convergent, whereas Newton's method will converge quadratically. The condition of convergence will be similar to \eqref{condition_Newton_constantK} for both methods. From the theoretical point of view, only a quantitatively increased robustness for the Picard method comparing with Newton's method should be expected, i.e. when e.g. the mesh size becomes smaller if one of the method fails then also the other (see Fig. \ref{fig:example_1_Iter_psi_vad_3}, where Newton's methods is not converging and modified Picard converges, but increasing the number of elements leads to divergence for the modified Picard or Picard/Newton methods as well). This is not the case with the $L-$scheme, which is clearly the most robust out of the considered methods, see Section \ref{sec:numerics}.
\end{remark}

\begin{remark} \label{remark:adaptivity} By using error estimates derived as mentioned in the remark above, one can construct an indicator to predict the convergence of Newton's method. Based on this, one can design an adaptive algorithm for using the $L-$scheme only when necessary. Nevertheless, because the $L-$iterations are so cheap and the resulting linear systems are (much) better conditioned, it seems that the $L-$scheme/Newton is almost that fast as the Newton method. In Example 1 in Section  \ref{sec:numerics} we even experienced that the $L-$scheme/Newton was faster than the Newton method. Therefore, we simply recommend the use of the $L-$scheme/Newton with a fixed number of $L-$iterations (4-5), without any indicator predictions. It the case of convergence failing, one should as a response automatically increase the number of $L-$iterations. We never experienced the need of more than 10 $L-$iterations in order to guarantee the convergence of the $L-$scheme/Newton.
\end{remark}

\section{Numerical results}\label{sec:numerics}
In this section, numerical results in two spatial dimensions are presented. The considered linearization schemes: the Newton method, the modified Picard method, Picard/Newton, the $L-$scheme and the $L-$scheme/Newton are comparatively studied. We focus on convergence, computational time and the condition number of the underlying linear systems. We consider two main numerical examples, both based on realistic parameters. The first one was developed by us, the second is a benchmark problem from \cite{Schneid}. Different conditions are created by varying the parameters. The sensitivity of the schemes w.r.t. the mesh size $h$ is particularly studied. All computations were performed on a Schenker XMG notebook with an Intel Core i7-3630GM processor. 

The relationships $K(\Psi)$ and $\theta(\Psi)$ for both examples are provided by the van Genuchten--Mualem model, namely
\begin{equation}
\begin{aligned}
\theta(\Psi) &=
\begin{cases}
\hspace{-.0cm} \theta_R + (\theta_S - \theta_R) \left[\frac{1}{1+(-\alpha \Psi)^n}\right]^{\frac{n-1}{n}}, & \Psi \leq 0, \\
\hspace{-.0cm} \theta_S, & \Psi > 0,
\end{cases} \\
K(\Psi) &= 
\begin{cases}
\hspace{-.0cm} K_S \theta(\Psi)^{\frac{1}{2}} \left[1-\left(1-\theta(\Psi)^{\frac{n}{n-1}}\right)^{\frac{n-1}{n}}\right]^2, & \Psi \leq 0, \\
\hspace{-.0cm} K_S, & \Psi > 0,
\end{cases}
\end{aligned}
\label{eq:Van_Genuchten}
\end{equation}
in which $\theta_S$ and $K_S$ denote the water content respectively the hydraulic conductivity when the porous medium is fully saturated, $\theta_R$ is the residual water content and $\alpha$ and $n$ are model parameters related to the soil properties. We compute the derivatives of $K$ and $\theta$ analytically whenever they arise. The evaluation of integrals is executed by applying a quadrature formula accurate for polynomials up to a degree of 4. 

\begin{remark}
The use of automatic differentiation might speed up the Newton method, but the concerns regarding the robustness will remain. This and the fact that most of the codes for solving Richards'	 equation do not have implemented automatic differentiation, were the reasons to compute the derivatives as mentioned above.
\end{remark}

\subsection{Example 1}
\label{sec:example_1}
This example deals with injection and extraction in the vadose zone $\Omega_{\mathrm{vad}}$ located above the groundwater zone $\Omega_{\mathrm{gw}}$. The composite flow domain is $\Omega = \Omega_{\mathrm{vad}} \cup \Omega_{\mathrm{gw}}$ defined as $\Omega_{\mathrm{vad}} = (0,1) \times (-3/4,0) $ and $\Omega_{\mathrm{gw}} = (0,1) \times (-1,-3/4]$. We choose the van Genuchten parameters $\alpha=0.95$, $n=2.9$, $\theta_S=0.42$, $\theta_R=0.026$ and $K_S=0.12$ in parametrization \eqref{eq:Van_Genuchten}. The choice $n > 2$ implies Lipschitz continuity of both $\theta$ and $K$. Constant Dirichlet conditions $\Psi \equiv -3$ on the surface $\Gamma_D = (0,1) \times \{0\}$ and no-flow Neumann conditions on $\Gamma_N = \partial \Omega \setminus \Gamma_D$ are imposed. The initial pressure height distribution is discontinuous at the transition of the groundwater to the vadose zone and is given by $\Psi^0 \equiv \Psi_{\mathrm{vad}}$ on $\Omega_{\mathrm{vad}}$ and $\Psi^0 = \Psi^0(z) = -z - 3/4$ on $\Omega_{\mathrm{gw}}$. We investigate two initial pressure heights in the vadose zone, $\Psi_{\mathrm{vad}} \in \{-3, -2\}$. In the vadose zone, we select a source term taking both positive and negative values given by $f = f(x,z) = 0.006 \cos(4/3 \pi z) \sin(2 \pi x)$ on $\Omega_{\mathrm{vad}}$, whereas we have $f \equiv 0$ in the saturated zone $\Omega_{\mathrm{gw}}$. 

We examine the numerical solutions after the first time step for $\tau = 1$. A regular mesh is employed, consisting of right-angled triangles whose legs are of length $h = \Delta x = \Delta z$ for $h \in \{\frac{1}{10},\frac{1}{20},\frac{1}{30},\frac{1}{40},\frac{1}{50},\frac{1}{60}\}$ (the mesh size is actually $h\sqrt{2}$). The parameters regulating the switch for the mixed methods are taken as $\delta_a = 2$ and $\delta_r = 0$. The computation using the $L-$scheme was carried out with parameter $L$ slightly greater than $L_\theta = \sup_{\Psi} \theta'(\Psi) = 0.2341$ for the given van Genuchten parametrization, to be specific $L = 0.25$.  However, as pointed out in the analysis, when the influence of the nonlinear $K$ is not that big (see Remark \ref{remark:noK}), a constant $L$ bigger than $\dfrac{L_\theta}{2}$ is enough for the convergence. According to our experience, this is the limit relevant for the practice. Hence, we performed another computation with parameter $L = 0.15$. For the mixed $L-$scheme/Newton we choose $L = 0.15$ as well.

The results for Example 1 are presented in Figs. \ref{fig:example_1_Iter_psi_vad_2} - \ref{fig:example_1_cond_psi_vad_3} and discussed in detail below.

\subsubsection{Convergence}
In case of higher initial moisture in the vadose zone, that is $\Psi_{\mathrm{vad}} = -2$, convergence was observed for all methods and all investigated meshes. For the choice $\Psi_{\mathrm{vad}} = -3$, Newton's method failed on each mesh, the modified Picard scheme exhibited convergence only for $h \geq \frac{1}{40}$, whereas both parametrizations of the $L-$scheme converged on all meshes. This is consistent with the theoretical findings in Section \ref{sec:convergence}, in particular with Remark \ref{remark:convergenceNewton}.

\subsubsection{Numbers of iterations}
The required numbers of iterations are depicted in Figs. \ref{fig:example_1_Iter_psi_vad_2} and \ref{fig:example_1_Iter_psi_vad_3}. Missing markers indicate that the iteration has not converged. For either value of $\Psi_{\mathrm{vad}}$, the smaller parameter $L = 0.15$ in the $L-$scheme yielded the criterion for convergence to be fulfilled after fewer iterations than $L = 0.25$. 

For $\Psi_{\mathrm{vad}} = -2$, the modified Picard scheme required less iterations than the $L-$scheme on coarse meshes, but for $h \leq \frac{1}{40}$, it needed at least as many iterations as the $L-$scheme. Newton's method featured an even smaller number of iterations which was found to be independent of the mesh size in our computation. The number of iterations for the mixed Picard/Newton scheme did not differ significantly from the one for Newton's method, while the mixed $L-$scheme/Newton needed the least iterations. 

For $\Psi_{\mathrm{vad}} = -3$, the modified Picard scheme had a benefit over the $L-$scheme in view of the number of iterations whenever it converged, although the number of iterations increased considerably as the mesh became finer. The mixed schemes gave the best results with respect to the number of iterations, the application of the mixed Picard/Newton scheme however being limited to coarse meshes.

\begin{figure}[h!]
\begin{center}
\includegraphics[scale=.18]{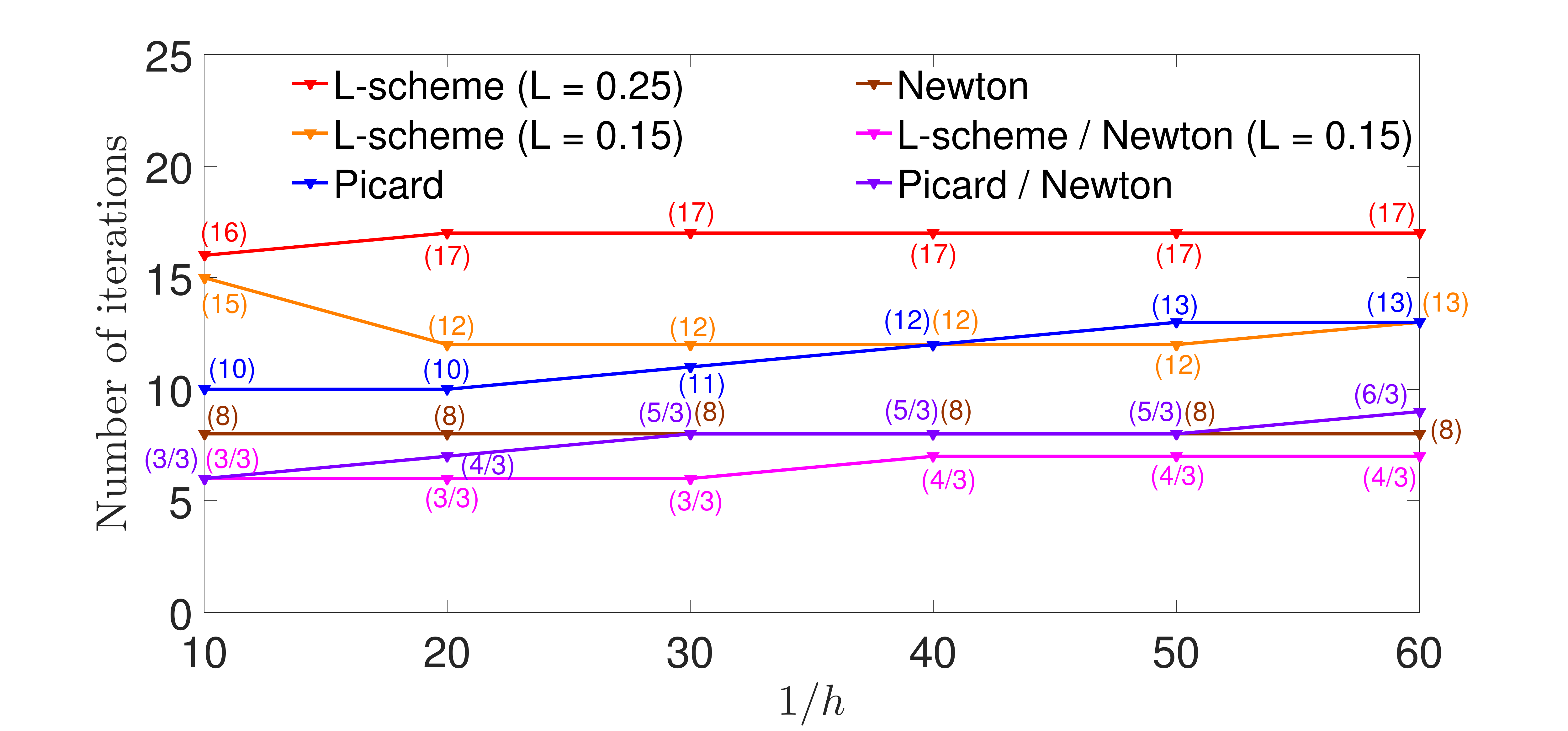}
\caption{Numbers of iterations for several mesh sizes, $\Psi_{\mathrm{vad}} = -2$}
\label{fig:example_1_Iter_psi_vad_2}
\end{center}
\end{figure}
\begin{figure}[h!]
\begin{center}
\includegraphics[scale=.18]{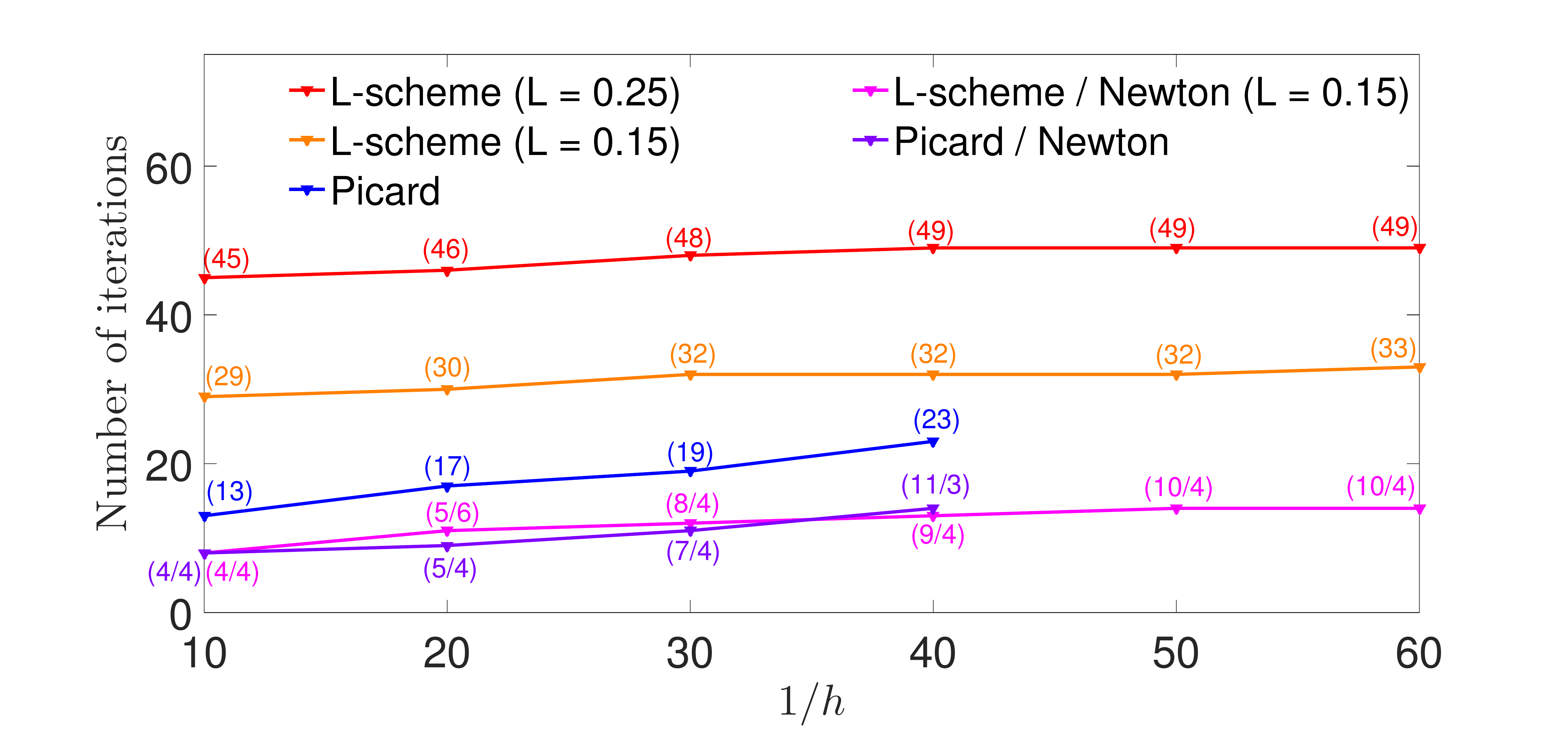}
\caption{Numbers of iterations for several mesh sizes,  $\Psi_{\mathrm{vad}} = -3$}
\label{fig:example_1_Iter_psi_vad_3}
\end{center}
\end{figure}

\subsubsection{Computation times}
Fig. \ref{fig:example_1_CPU_psi_vad_2} shows the simulation times for $\Psi_{\mathrm{vad}} = -2$. Although the modified Picard scheme needed less iterations than the $L-$scheme with $L = 0.25$, the differences of computation times were small, since the modified Picard scheme requires the computation of matrices including $\theta'(\Psi)$. For Newton's method, $K'(\Psi)$ has to be calculated in addition. Nevertheless, it converged more rapidly than the modified Picard scheme. As reported by \cite{LehmannAckerer}, combination of the modified Picard scheme and Newton's method further improved the performance in terms of computation time. However, both $L-$scheme with $L = 0.15$ and mixed $L-$scheme/Newton exhibited faster convergence than the mixed Picard scheme on dense grids, the mixed $L-$scheme/Newton only taking $71.2\%$ of computation time compared to the mixed Picard/Newton scheme for $h = \frac{1}{60}$. 

The simulation times for $\Psi_{\mathrm{vad}} = -3$ are presented in Fig. \ref{fig:example_1_CPU_psi_vad_3}. The mixed schemes computed the solution faster than the non-mixed schemes on each mesh, the mixed $L-$scheme/Newton taking roughly half the computation time in comparison to the non-mixed $L-$scheme with $L = 0.15$.

\begin{figure}[h!]
\begin{center}
\includegraphics[scale=.18]{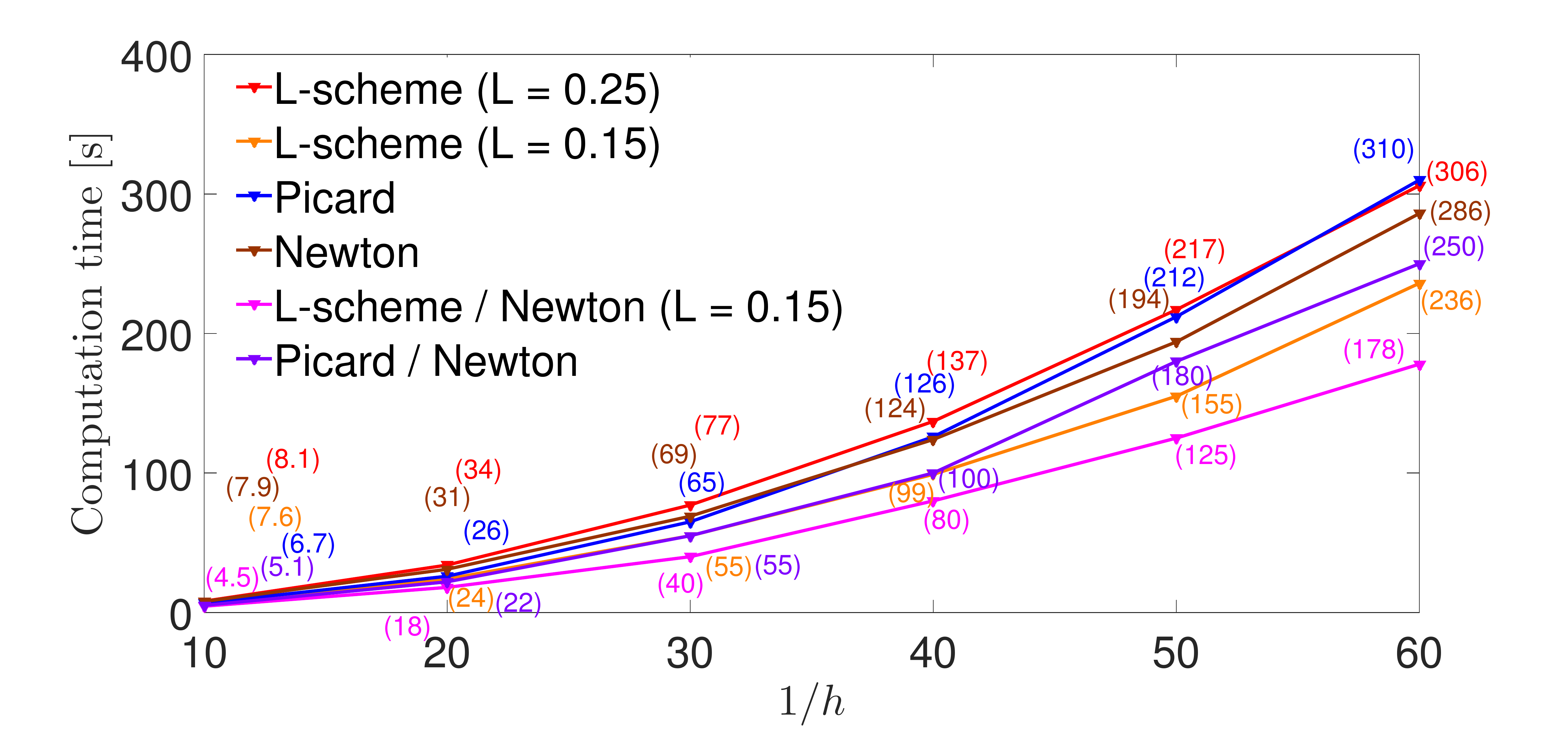}
\caption{Computation times for several mesh sizes,  $\Psi_{\mathrm{vad}} = -2$}
\label{fig:example_1_CPU_psi_vad_2}
\end{center}
\end{figure}
\begin{figure}[h!]
\begin{center}
\includegraphics[scale=.18]{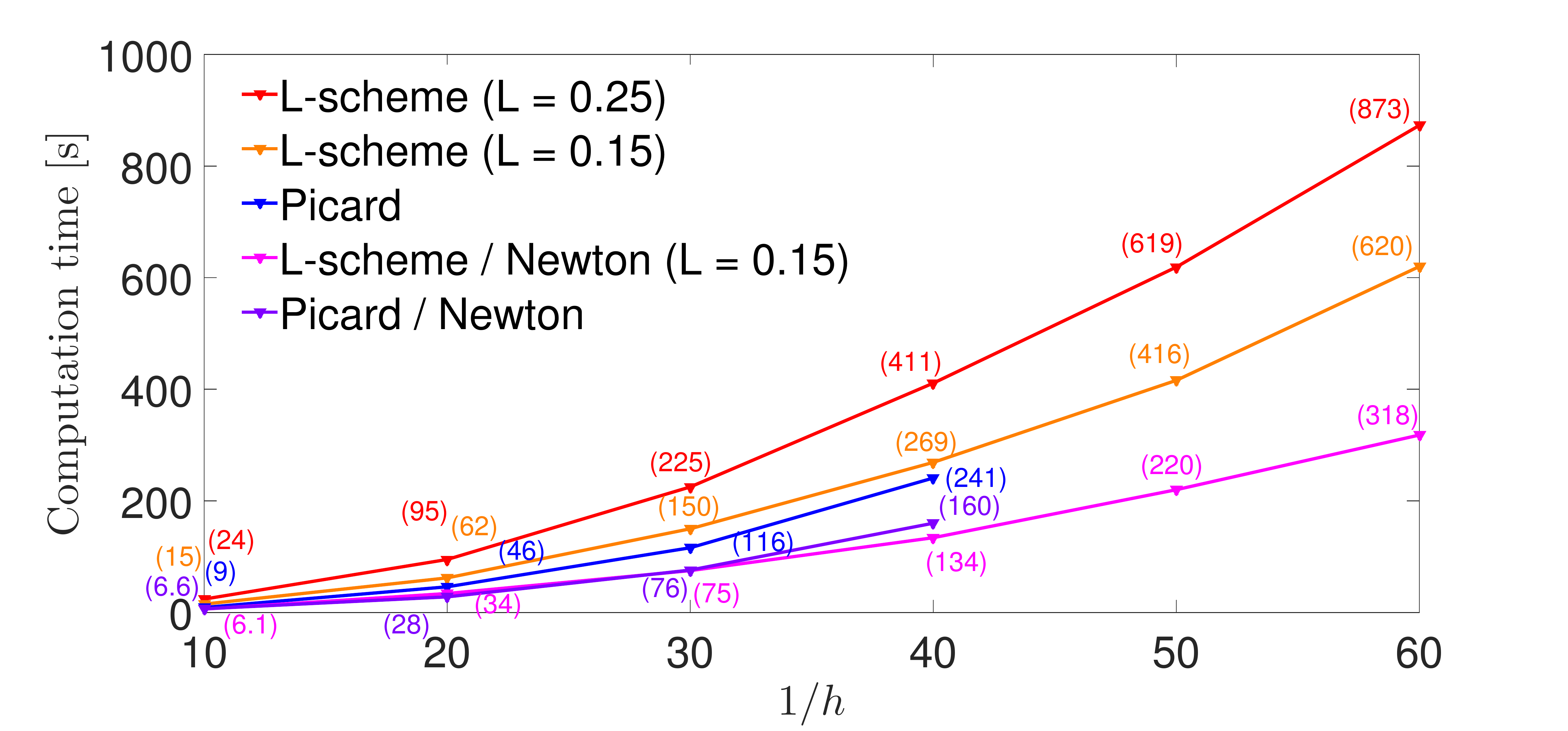}
\caption{Computation times for several mesh sizes,  $\Psi_{\mathrm{vad}} = -3$}
\label{fig:example_1_CPU_psi_vad_3}
\end{center}
\end{figure}

\subsubsection{Condition numbers}
In light of the accuracy of the numerical results, it is interesting to examine the condition numbers of the left-hand side matrices in the system of linear equations for the coefficient vector. Estimations for the condition numbers with respect to $L^1(\Omega)$, denoted by $\|\cdot\|_1$ calculated using the MATLAB function condest() are plotted in Figs. \ref{fig:example_1_cond_psi_vad_2} and \ref{fig:example_1_cond_psi_vad_3} for the non-mixed methods, averaged over all iterations. They did hardly differ from each other at several iteration steps and condition numbers for the mixed methods corresponded approximately to the ones of the respective non-mixed method in each iteration. For both values of $\Psi_{\mathrm{vad}}$, the $L-$scheme with $L = 0.25$ featured the lowest condition numbers, followed by its counterpart with $L = 0.15$. In case of Newton's method being convergent, it exhibited higher condition numbers than the $L-$scheme. In all computations, the condition numbers in the modified Picard scheme were the highest, furthermore, they increased most rapidly when it came to denser meshes.

All methods required more iterations and computation time when the vadose zone was taken to be dryer initially and the arising matrices were worse-conditioned than for the moister setting.

\begin{figure}[h!]
\begin{center}
\includegraphics[scale=.18]{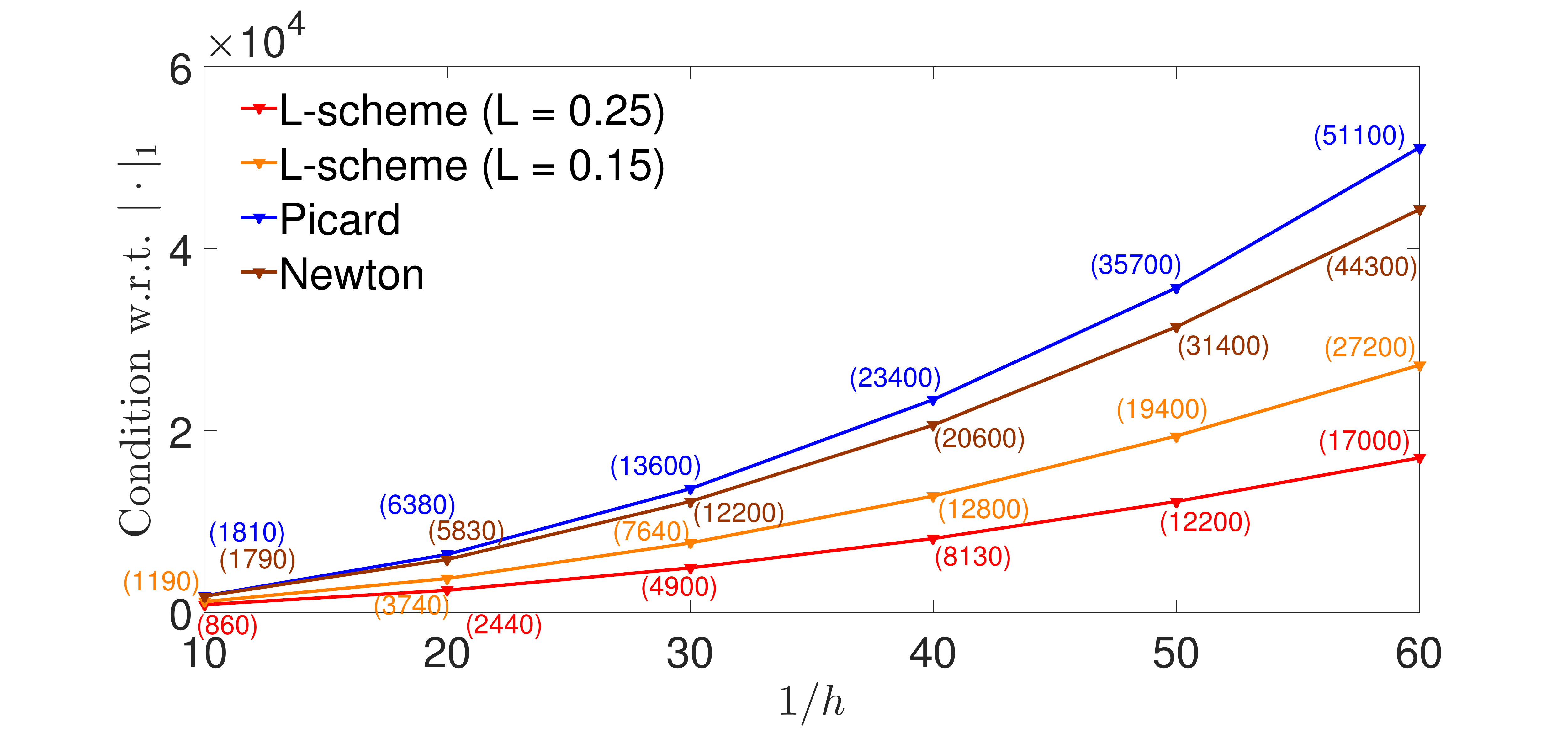}
\caption{Condition numbers for several mesh sizes,  $\Psi_{\mathrm{vad}} = -2$}
\label{fig:example_1_cond_psi_vad_2}
\end{center}
\end{figure}
\begin{figure}[h!]
\begin{center}
\includegraphics[scale=.18]{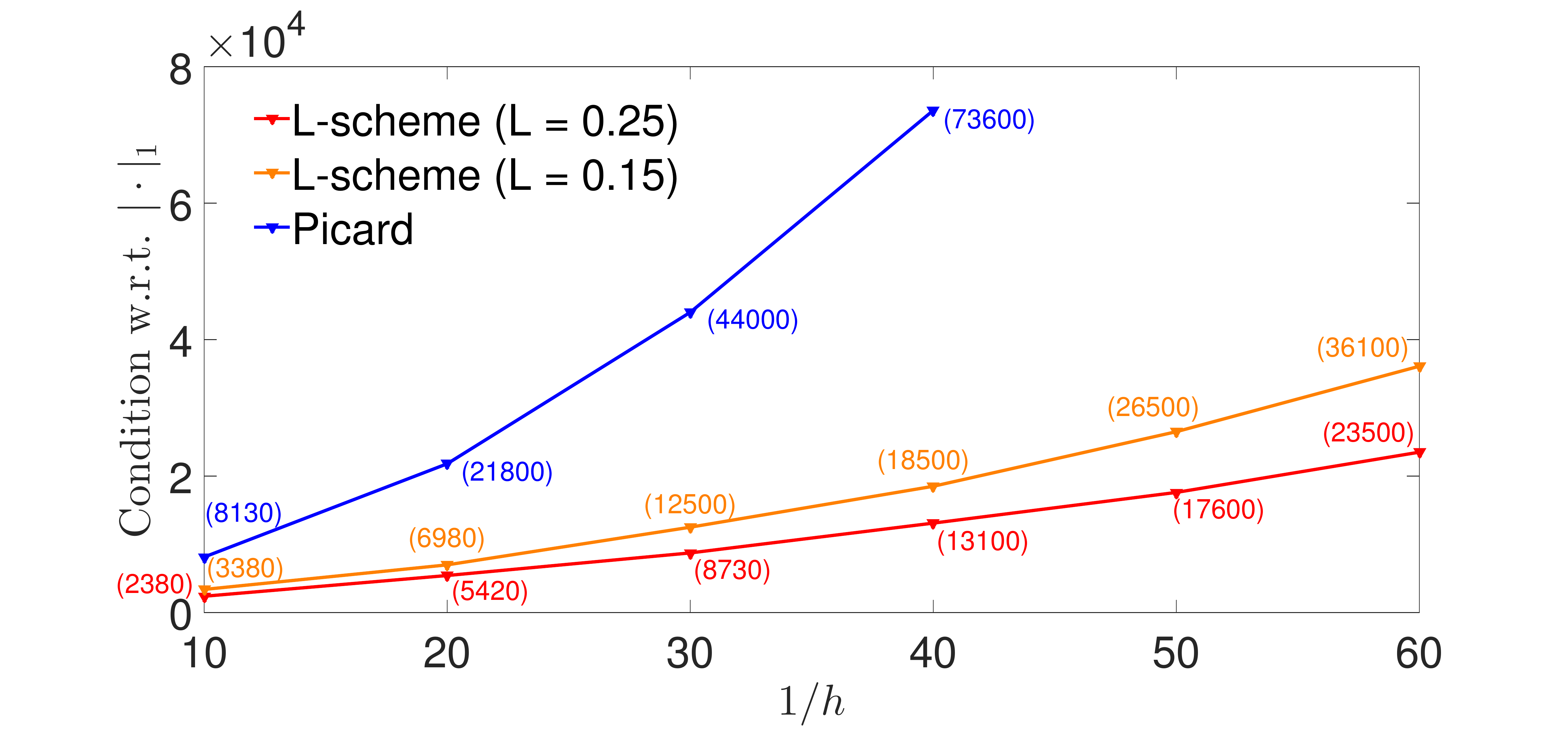}
\caption{Condition numbers for several mesh sizes,  $\Psi_{\mathrm{vad}} = -3$}
\label{fig:example_1_cond_psi_vad_3}
\end{center}
\end{figure}

\subsection{Example 2 (Benchmark problem)}
\label{sec:example_2}
In order to compare the linearization methods in the numerical simulation of a recognized benchmark problem, we consider an example used by \cite{Haverkamp}, \cite{Knabner} and \cite{Schneid} amongst others. It describes the recharge of a groundwater reservoir from a drainage trench in two spatial dimensions (Fig. \ref{fig:geometry_ex2}). The domain $\Omega \subset \mathbb{R}^2$ represents a vertical section of the subsurface. On the right hand side of $\Omega$, the groundwater table is fixed by a Dirichlet condition for the pressure height for $z \in [0,1]$. The drainage trench is modelled by a transient Dirichlet condition on the upper boundary for $x \in [0,1]$. On the remainder of the boundary $\partial \Omega$, no-flow conditions are imposed. Hence, the left boundary can be construed as symmetry axis of the geometry and the lower boundary as transition to an aquitard. Altogether, the geometry is given by
\begin{equation*}
\begin{aligned}
\Omega &= (0,2) \times (0,3),  \\
\Gamma_{D_1}    &= \{(x,z) \in \partial \Omega \ | \ x \in [0,1] \land z = 3 \}, \\
\Gamma_{D_2}    &= \{(x,z) \in \partial \Omega \ | \ x = 2 \land z \in [0,1] \}, \\
\Gamma_D        &= \Gamma_{D_1} \cup \Gamma_{D_2}, \\
\Gamma_N        &=\partial \Omega \setminus \Gamma_D. \\
\end{aligned}
\end{equation*}
The initial and boundary conditions are taken as
\begin{equation*}
\begin{aligned}
\Psi(x,z,t) &=
\begin{cases} -2+2.2 \ t/\Delta t_D, & \mathrm{on} \ \Gamma_{D_1}, t \leq \Delta t_D, \\ 
0.2, & \mathrm{on} \ \Gamma_{D_1}, t > \Delta t_D,  \\
1-z, & \mathrm{on} \ \Gamma_{D_2}, \end{cases} \\
-K(\Psi(x,z,t))&(\nabla \Psi(x,z,t) + \ez \cdot \vec{n}) = 0 \ \hspace{.2cm} \mathrm{on} \ \Gamma_N, \\
&\Psi^0(x,z) = 1 - z  \ \hspace{.97cm} \mathrm{on} \ \Omega,\\
\end{aligned}
\end{equation*}
in which $\vec{n}$ denotes the outward pointing normal vector. Initially, a hydrostatic equilibrium is thus assumed. The computations are undertaken for two sets of parameters adopted from \cite{Genuchten}, characterising silt loam respectively Beit Netofa clay. For both soil types, the solution is computed over $N = 9$ time levels. The time unit is $1$ day and dimensions are given in meters. The van Genuchten parameters employed as well as the parameter $\Delta t_D$ governing the time evolution of the upper Dirichlet boundary, the time step $\tau$ and the simulation end time $T$ are listed in Table  \ref{table_parameters_ex2}. We used a regular mesh consisting of $651$ nodes. The simulations invoking the $L-$scheme were carried out with $L = \sup_{\Psi} \theta'(\Psi)$ (referred to as $L-$scheme 1) and with $L$ slightly smaller (referred to as $L-$scheme 2) for both soil types, that is $L = 4.501 \cdot 10^{-2}$ and $L = 3.500 \cdot 10^{-2}$ for the silt loam soil and $L = 7.4546 \cdot 10^{-3}$ and $L = 6.500 \cdot 10^{-3}$ for the clay soil. The mixed methods switched to Newton's method when condition \eqref{eq:criterion_switch} held true for $\delta_a = 0.2$ and $\delta_r = 0$. 

\begin{figure}[h!]
\begin{center}
\includegraphics[scale=0.5]{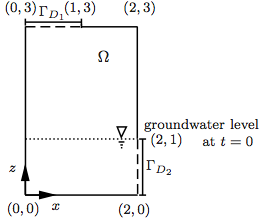}
\end{center}
\caption{Geometry for Example 2}
\label{fig:geometry_ex2}
\end{figure}

\begin{table}[h!]
\begin{center}
\begin{tabularx}{.38\textwidth}{||l|r|r||} 
\toprule
& Silt loam & Beit Netofa clay \\
\toprule
\multicolumn{3}{l}{\emph{Van Genuchten parameters:}}\\
\midrule
$\theta_S$                                               & $0.396$ &  $0.446$\\
$\theta_R$                                               & $0.131$ & $0.0$\\
$\alpha$                                                 & $0.423$ & $0.152$\\
$n$                                                      & $2.06$ & $1.17$\\
$K_S$                                                    & $4.96 \cdot 10^{-2}$ & $8.2 \cdot 10^{-4}$\\
\midrule
\multicolumn{3}{l}{\emph{Time parameters:}}\\
\midrule
$\Delta t_D$ 		  				           & $1/16$ & $1$\\
$\Delta t$ 		  				               & $1/48$ & $1/3$\\
$T$                                            & $3/16$ & $3$\\
\bottomrule
\end{tabularx}
\caption{Simulation parameters for Example 2}
\label{table_parameters_ex2}
\end{center}
\end{table}

All the considered linearization methods converged for both soil types. The pressure profiles computed with mixed $L-$scheme 2/Newton at time $T$ are presented in Fig. \ref{fig:solution_plot_silt} and are as expected for this benchmark problem. Table \ref{tab:comparison} shows the total numbers of iterations, the computation times and the average of the estimated condition numbers of the left-hand side matrices with respect to $\|\cdot\|_1$, in case of mixed methods split up in the two involved schemes. In what follows, the foregoing numerical indicators, i.e. the number of iterations, the computational time and the condition numbers are to be discussed in detail.

\begin{figure}[h!]
\hspace*{-2cm}
\includegraphics[scale=.2]{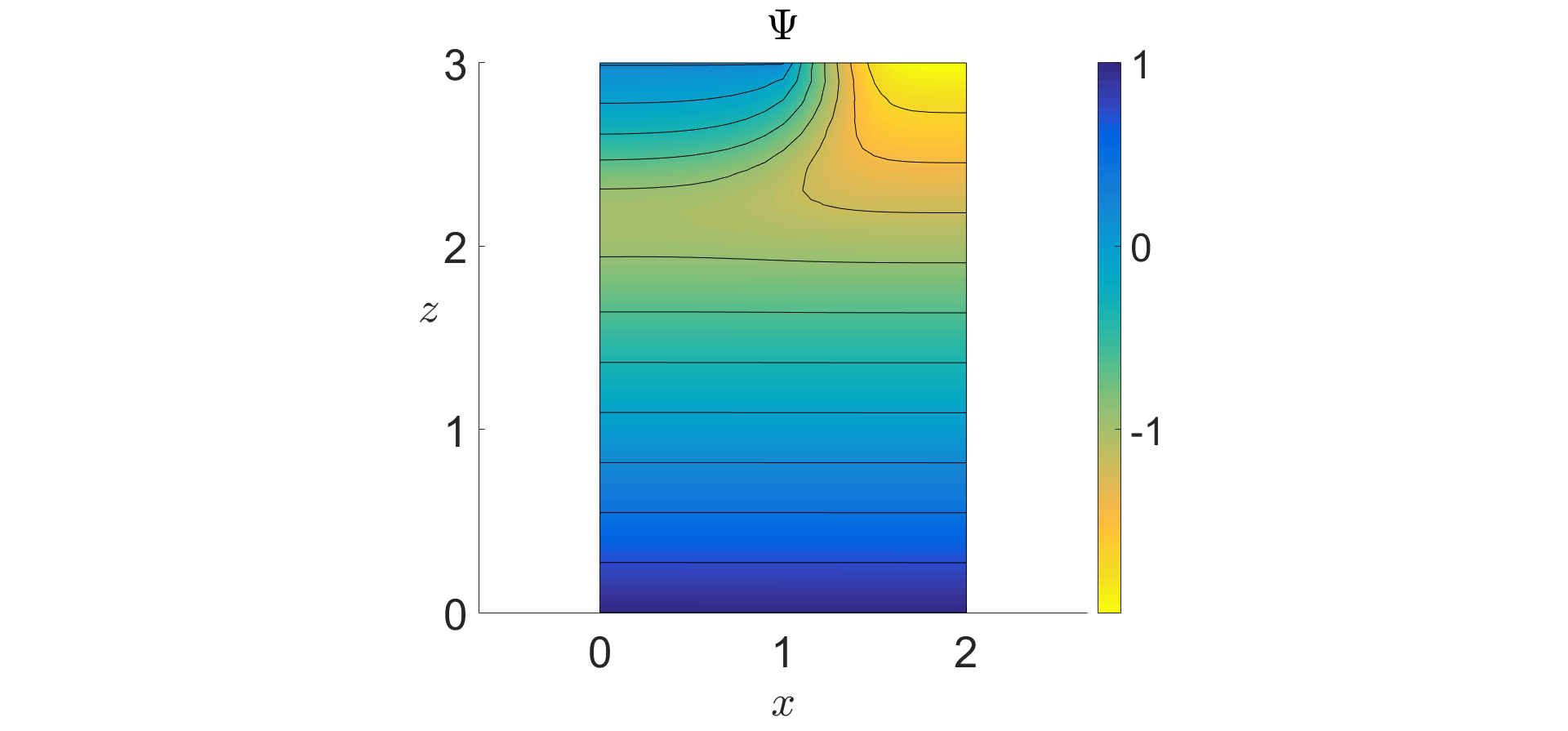}\mbox{ }\hspace*{-2.6cm} 
\includegraphics[scale=.2]{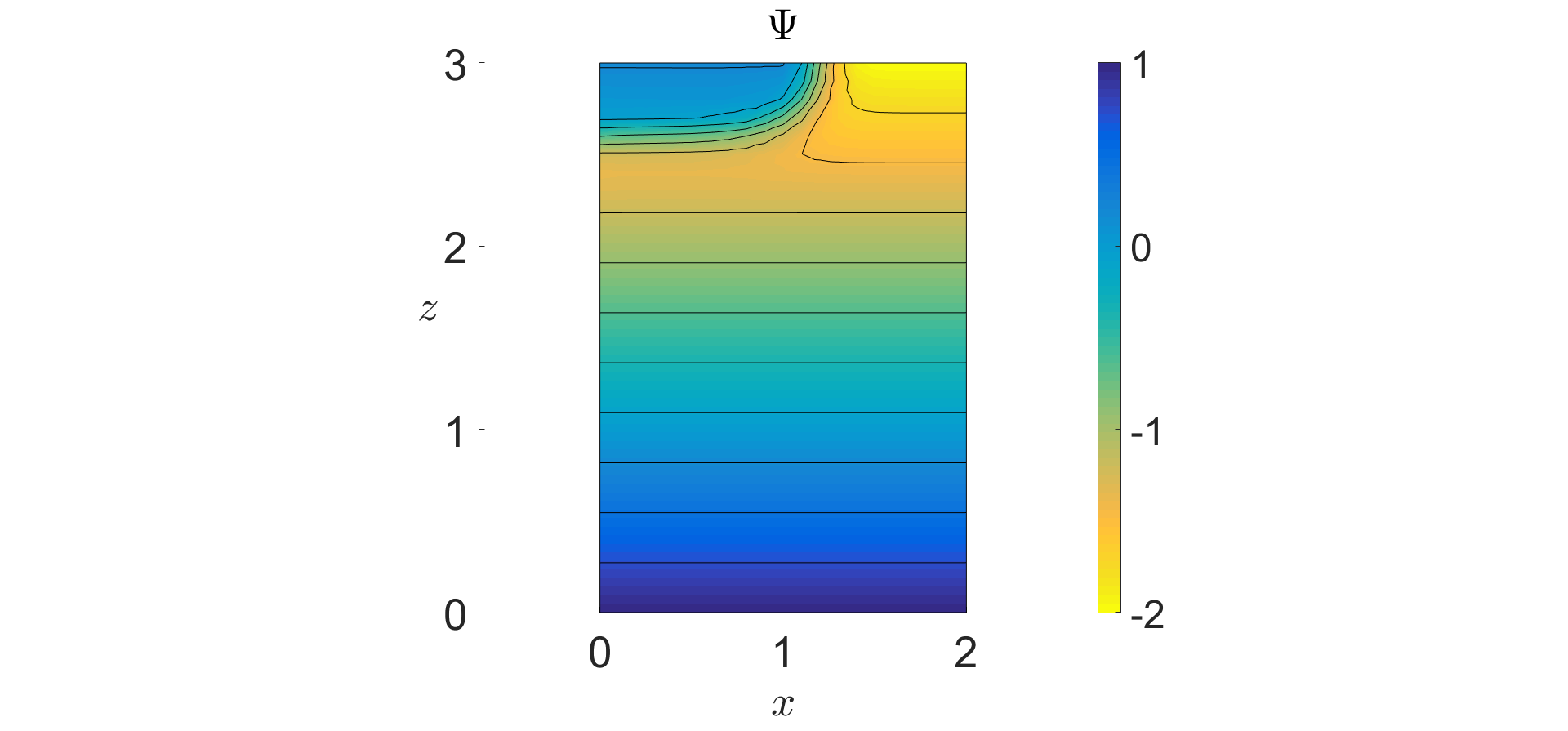}
\caption{Pressure profiles after $4.5 \ [h]$ for silt loam (left) and $3 \ [d]$ for Beit Netofa clay (right)}
\label{fig:solution_plot_silt}
\end{figure}

\subsubsection{Numbers of iterations}
As to the non-mixed methods, it is not surprising that more complex methods yielded smaller numbers of iterations, i.e. Newton's method converged after the fewest iterations, followed by the modified Picard scheme. $L-$scheme 2 had the edge over $L-$scheme 1, but still needed some more iterations than the modified Picard scheme for both soil types. The numbers of iterations of the mixed methods exhibit a salient result: The advantage of the modified Picard scheme over the $L-$scheme with regard to the number of iterations vanished when coupling the schemes to Newton's method and the mixed $L-$scheme 2/Newton required less iterations than the mixed Picard/Newton scheme. This suggests that the $L-$scheme stands out due to a rapid approach towards the solution in the first iteration steps. Among all methods, Newton's method provided convergence after the least number of iterations for both van Genuchten parametrizations.

\subsubsection{Computation times}
When it comes to the comparison of computation times, it is striking that the performances of the methods substantially varied between the simulations for silt loam and Beit Netofa clay. While Newton's method featured the shortest computation time among the non-mixed methods in case of silt loam owing to the low number of required iterations, computation in case of the clayey soil took long using Newton's method as compared to the $L-$scheme. In the silt loam simulation, computation times of the $L-$scheme were clearly greater than the ones of Newton's method, but switching to Newton's method vastly improved the computation time so that the $L-$scheme 2/Newton turned out to be the fastest method. In contrast, the computation times for the clay soil demonstrate that in some cases, switching to Newton's method may even be disadvantageous. Although the mixed $L-$scheme/Newton converged in fewer iteration steps than the non-mixed ones, changing to Newton's method provoked a deterioration of the computation time. This might indicate that the $L-$scheme be less susceptible to parametrizations of the hydraulic relationships lacking of regularity than the modified Picard scheme and Newton's method since the hydraulic conductivity for the parametrization of the Beit Netofa clay is not Lipschitz continuous. The modified Picard scheme was found to be the slowest method for the silt loam soil, the computation time for Beit Netofa clay was barely less than the one related to Newton's method.

\subsubsection{Condition numbers}
In view of the condition numbers of the left-hand side matrices, the $L-$scheme excelled for both soil types: The condition numbers with either value of $L$ were remarkably lower than the ones arising when Newton's method or the modified Picard scheme were employed, to be more specific by a factor of minimum 11 for the silty soil and still by a factor of minimum 5 for the clayey soil. Apparently, incorporation of the derivative of the water content entailed a considerable deterioration of the condition. The virtual equality of the condition numbers for the modified Picard scheme and Newton's method was probably due to the proximity of the solution to a hydrostatic equilibrium which caused the only term distinguishing Newton's method from the modified Picard scheme in equation \eqref{eq:Newton} to be small because of $\nabla \Psi_h^{n} \approx -\ez$.

\begin{table}[h!]
\begin{center}
\medskip\noindent
\begin{tabularx}{0.53\textwidth}{||l|c|c||}
\toprule
& Silt loam & Beit Netofa clay \\
\toprule
\multicolumn{3}{l}{\emph{Total number of iterations:}}\\
\midrule
$L-$scheme 1 & 74 & 74 \\
$L-$scheme 2 & 65 & 72 \\
Picard & 58 & 69 \\
Newton & 31 & 48 \\
$L-$scheme 1 / Newton & 46 (26/20) & 54 (28/26) \\
$L-$scheme 2 / Newton & 40 (22/18) & 54 (28/26) \\
Picard / Newton & 43 (25/18) & 55 (29/26) \\
\midrule
\multicolumn{3}{l}{\emph{Total computation time $[s]$:}}\\
\midrule
$L-$scheme 1 & 231 & 237 \\
$L-$scheme 2 & 210 & 225 \\
Picard & 234 & 285 \\
Newton & 184 & 289 \\
$L-$scheme 1 / Newton & 200 & 247 \\
$L-$scheme 2 / Newton & 180 & 243 \\
Picard / Newton & 213 & 278 \\
\midrule
\multicolumn{3}{l}{\emph{Averaged condition number $[10^3]$}:}\\
\midrule
$L-$scheme 1 & 6.84 & 51.2 \\
$L-$scheme 2 & 7.86 & 56.0 \\
Picard & 90.1 & 321 \\
Newton & 90.1 & 321 \\
$L-$scheme 1 / Newton & 6.84/90.1 & 51.2/321 \\
$L-$scheme 2 / Newton & 7.86/90.1 & 56.0/321 \\
Picard / Newton & 9.01/90.1 & 321/321 \\
\bottomrule
\end{tabularx}
\caption{Comparison of the linearization methods for Example 2}
\label{tab:comparison}
\end{center}
\end{table}


\section{Conclusions}\label{sec:conclusions}

In this paper we considered linearization methods for the Richards' equation. The methods were comparatively studied w.r.t. convergence, computational time and condition number of the resulting linear systems. The analysis was done in connection with Galerkin finite elements, but the schemes can be applied to any other discretization method as well, and similar results are expected. We focused on the Newton method, the modified Picard method, the Picard/Newton and the $L-$scheme. We proposed also a new mixed scheme, the $L-$scheme/Newton which seems to perform best. We conducted a theoretical analysis for the $L-$scheme for Richards' equation, showing that it is robust and linearly convergent. We also discussed the convergence of the modified Picard and Newton methods. 

The $L-$scheme is very easy to be implemented, does not involve the computation of any derivatives and the resulting linear systems are much better conditioned as the modified Picard or Newton methods. Although it is only linearly convergent, seems to be not much slower than the Newton (or Picard/Newton) method, and in some cases even faster. The $L-$scheme is the only robust one, a result which can be shown theoretically and it is supported by the numerical findings. Only a relatively mild constraint on the time step length is required. Furthermore, when the hydraulic conductivity $K$ is a constant, there is no restriction in the time step size. In this case the only condition necessary for the global convergence of the $L-$method is $L \ge \dfrac{L_\theta}{2}$.

We proposed a new mixed scheme, the $L-$scheme/Newton which is more robust than Newton but still quadratically convergent. This new mixed method performed best from all the considered methods with respect to computational time. Even in cases when Newton converges, the $L-$scheme/Newton seems to be worth, being faster for the examples considered.

The present study is based on two illustrative numerical examples, with realistic parameters. The examples are two dimensional. One of the examples is a known benchmark problem. The numerical findings are sustaining the theoretical analysis.\\

{\bf Acknowledgements}. Thank the DFG-NRC NUPUS for support and our colleagues J.M. Nordbotten, I.S. Pop and K. Kumar for very helpful discussions.

\end{document}